\documentclass[12pt, reqno]{amsart}
\usepackage{amsmath}
\usepackage{amssymb}
\usepackage{amsfonts}
\usepackage{amsthm}
\usepackage{amscd}
\usepackage{array}
\usepackage{relsize}
\usepackage{setspace}
\usepackage{geometry}
\usepackage{enumerate}
\usepackage{url}
\usepackage{xspace}
\usepackage{mathrsfs}
\usepackage{dsfont}
\usepackage{lmodern}
\usepackage{xcolor}
\usepackage[utf8]{inputenc}
\usepackage{mathtools}
\usepackage{tabu}
\usepackage[all]{xy}

\usepackage[colorlinks,linkcolor={red},citecolor={blue}]{hyperref}

\usepackage[textsize=footnotesize,textwidth=15ex]{todonotes}

\newtheorem{theorem}{Theorem}[section]
\newtheorem{lemma}[theorem]{Lemma}
\newtheorem{prop}[theorem]{Proposition}
\newtheorem{corollary}[theorem]{Corollary}

\theoremstyle{definition}
\newtheorem{remark}[theorem]{Remark}
\newtheorem{example}[theorem]{Example}
\newtheorem{definition}[theorem]{Definition}
\newtheorem{question}[theorem]{Question}

\newcommand{\BBB}{\mathcal{B}}

\newcommand{\UU}{\mathcal{U}}

\newcommand{\CC}{\mathbb{C}}
\newcommand{\FF}{\mathbb{F}}
\newcommand{\KK}{\mathbb{K}}
\newcommand{\NN}{\mathbb{N}}

\newcommand{\ZZ}{\mathbb{Z}}

\newcommand{\g}{\mathfrak{g}}

\newcommand{\h}{\mathfrak{h}}
\newcommand{\n}{\mathfrak{n}}
\newcommand{\sll}{\mathfrak{sl}}

\newcommand{\G}{\mathfrak{G}}

\newcommand{\U}{\mathfrak{U}}

\DeclareMathOperator{\ad}{ad}
\DeclareMathOperator{\Ad}{Ad}
\DeclareMathOperator{\Aut}{Aut}
\DeclareMathOperator{\charact}{char}

\DeclareMathOperator{\End}{End}
\DeclareMathOperator{\Ga}{\mathbb{G}_a}
\DeclareMathOperator{\Gr}{Gr}

\DeclareMathOperator{\height}{ht}

\DeclareMathOperator{\SL}{SL}
\DeclareMathOperator{\SSL}{E}
\DeclareMathOperator{\Zalg}{{\mathbb Z}-alg}

\newcommand{\inv}{^{-1}}

\newcommand{\co}{\colon\thinspace}

\DeclareMathOperator{\GL}{GL}

\begin{document}

\title{Minicourse on Kac--Moody groups}

\author{Timoth\'{e}e \textsc{Marquis}$^*$}

\thanks{$^*$F.R.S.-FNRS Research associate; supported in part by the FWO and the F.R.S.-FNRS under the EOS programme (project ID 40007542).}

\begin{abstract}
These are informal lecture notes for a three-hour minicourse on Kac--Moody groups, given at the workshop ``Kac--Moody geometry'' in July 2023 in Kiel. They provide a concise overview of the book \emph{An introduction to Kac--Moody groups over fields}, EMS Textbooks in Mathematics (2018). They assume a previous familiarity with the (very) basics of Kac--Moody algebras. For readers unfamiliar with the latter topic, short ``Prerequisites'' notes (referenced within the text) are also freely available.
\end{abstract}

\maketitle

All results mentioned in these notes are contained in the \href{http://dx.doi.org/10.4171/187}{book} ``An introduction to Kac--Moody groups over fields''. So as to lighten the presentation, no bibliographical details are provided within the text. Instead, attributions of all mentioned results can be found at the end of Chapters 6, 7 and 8, and in Chapter 9 of the book. For readers unfamiliar with Kac--Moody algebras (and/or Coxeter groups, buildings and BN-pairs), short ``Prerequisites'' notes are also available \href{https://perso.uclouvain.be/timothee.marquis/papers/PrerequisitesKM.pdf}{here}.

\section{Setting}

\begin{itemize}
\item Let $A=(a_{ij})_{i,j\in I}$ be a generalised Cartan matrix (GCM).
\item Let $(\h,\Pi,\Pi^{\vee})$ be a realisation of $A$: $\h$ is a complex vector space with $\dim_{\CC}\h=|I|+\mathrm{corank}(A)$ and $\Pi=\{\alpha_i \ | \ i\in I\}$ and $\Pi^{\vee}=\{\alpha_i^{\vee} \ | \ i\in I\}$ are linearly independant subsets of $\h^*$ and $\h$ respectively such that $\langle\alpha_j,\alpha_i^{\vee}\rangle=a_{ij}$ for all $i,j\in I$.
\item Let $\g(A)=\langle \textrm{$e_i,f_i$ ($i\in I$), $\h$} \ | \ \textrm{(RA1)--(RA4)} \rangle$ be the Kac-Moody algebra of type $A$, where
\begin{itemize}
\item[(RA1)]
$[h,h']=0$ for all $h,h'\in\h$;
\item[(RA2)]
$[h,e_i]=\alpha_i(h)e_i$ and $[h,f_i]=-\alpha_i(h)f_i$ for all $h\in\h$ and $i\in I$;
\item[(RA3)]
$[f_i,e_j]=\delta_{ij}\alpha_i^{\vee}$ for all $i,j\in I$;
\item[(RA4)]
$(\ad e_i)^{|a_{ij}|+1}e_j=0=(\ad f_i)^{|a_{ij}|+1}f_j$ for all $i,j\in I$ with $i\neq j$.
\end{itemize}
\item We have a root space decomposition
$$\g(A)=\h\oplus\bigoplus_{\alpha\in\Delta}\g_{\alpha}\quad\textrm{where $\g_{\alpha}=\{x\in\g(A) \ | \ [h,x]=\alpha(h)x \ \forall h\in\h\}$}$$
and a triangular decomposition
$$\g(A)=\n^-\oplus\h\oplus\n^+$$
where $\n^+$ (resp. $\n^-$) is the subalgebra generated by the $e_i$ (resp. $f_i$), $i\in I$. We also consider the derived Kac--Moody algebra 
$$\g_A=\n^-\oplus\h'\oplus\n^+\quad\textrm{where $\h'=\sum_{i\in I}\CC\alpha_i^{\vee}\subseteq\h$.}$$
\item
The root system $\Delta=\{\alpha\in\h^*\setminus\{0\} \ | \ \g_{\alpha}\neq\{0\}\}$ decomposes into positive/negative roots:
$$\Delta=\Delta^+\sqcup \Delta^-\quad\textrm{where}\quad \Delta^{\pm}=\{\alpha=\pm\sum_{i\in I}n_i\alpha_i\in\Delta \ | \ n_i\in\NN\}.$$
For $\alpha$ as above, we write $\height(\alpha)=\pm\sum_{i\in I}n_i\in\ZZ$ for its height.
\item
The Weyl group 
$$W=\langle s_i\co \alpha\mapsto \alpha-\langle\alpha,\alpha_i^{\vee}\rangle \alpha_i \ | \ i\in I\rangle\leq\GL(\h^*)$$
of $A$ is a Coxeter group, with set of simple reflections $S=\{s_i \ | \ i\in I\}$.
\item
$\Delta$ also splits into the sets of real/imaginary roots:
$$\Delta=\Delta^{re}\sqcup\Delta^{im}\quad\textrm{where $\Delta^{re}=W\Pi$ and $\Delta^{im}=\Delta\setminus\Delta^{re}$}$$
and we set $\Delta^{re\pm}=\Delta^{re}\cap\Delta^{\pm}$ and $\Delta^{im\pm}=\Delta^{im}\cap\Delta^{\pm}$.
\end{itemize}

\begin{example}\label{example:gA}
\begin{enumerate}
\item
If $A=\begin{psmallmatrix}2&-1\\ -1&2\end{psmallmatrix}$, then $\g_A=\g(A)\cong\sll_3(\CC)$.
\item
If $A=\begin{psmallmatrix}2&-2\\ -2&2 \end{psmallmatrix}$ with $I=\{0,1\}$, then $\g_A\cong\sll_2(\CC[t,t\inv])\rtimes \CC K$ is a one-dimensional (nontrivial) central extension of $\sll_2(\CC[t,t\inv])$, with 
$$e_1=\begin{psmallmatrix}0&1\\ 0&0 \end{psmallmatrix}, \quad f_1=\begin{psmallmatrix}0&0\\ -1&0 \end{psmallmatrix}, \quad \alpha_1^{\vee}=\begin{psmallmatrix}1&0\\ 0&-1 \end{psmallmatrix}$$
and 
$$e_0=\begin{psmallmatrix}0&0\\ -t&0 \end{psmallmatrix}, \quad f_0=\begin{psmallmatrix}0&t\inv\\ 0&0 \end{psmallmatrix}, \quad \alpha_0^{\vee}=-\alpha_1^{\vee}+K.$$
\end{enumerate}
\end{example}

We conclude this section with the Gabber--Kac theorem, justifying why Kac--Moody algebras are generalisations of the \emph{simple} finite-dimensional complex Lie algebras, at least when $A$ is symmetrisable (i.e. $A=DB$ with $D$ a diagonal and $B$ a symmetric matrix).

\begin{theorem}[Gabber--Kac]\label{thm:GKac}
If $A$ is an indecomposable symmetrisable GCM, then $\g_A$ is simple modulo its center $\mathcal Z(\g_A)\subseteq\h'$.
\end{theorem}

\section{The world of Kac--Moody algebras}\label{section:KMA}
Call an (indecomposable) GCM $A$ of \emph{finite type} if $A=DB$ with $B$ positive definite, of \emph{affine type} if $A=DB$ with $B$ positive semidefinite and of corank $1$, and of \emph{indefinite type} otherwise.

\begin{center}
{\tabulinesep=2mm
\begin{tabu}{|m{4cm}|m{3cm}|m{4cm}|m{4cm}|}
\hline
&Finite type&Affine type&Indefinite type\\ \hline
$\dim\g_A$ &$<\infty$&$\infty$&$\infty$\\ \hline
growth as $n\to\infty$ of $\dim\bigoplus_{0<\height(\alpha)<n}\g_{\alpha}$&constant&polynomial&exponential\\ \hline
Known ``realisations''&e.g. $\sll_n(\CC)$& e.g. $\sll_n(\CC[t,t\inv])$ & ???\\ \hline
Coxeter group $W$& finite (spherical geometry)&affine (Euclidean geometry)& indefinite for $|I|\geq 3$ (hyperbolic-like geometry)\\ \hline
Roots &$\Delta^{re}=\Delta$ &$\Delta^{re}\subsetneq\Delta$, $\Delta^{im}=\ZZ_{\neq 0}\delta$& $\Delta^{re}\subsetneq\Delta$, $\Delta^{im}$ ``big''\\ \hline
$\g_{\alpha}$, $\alpha\in\Delta^{re}$&\multicolumn{3}{c|}{$\dim\g_{\alpha}=1$: choose\footnote{There is a canonical choice of $e_{\alpha}\in\g_{\alpha}$ (up to sign), defined using $W$ as in the ``Prerequisites'' notes.} $e_{\alpha}$ such that $\g_{\alpha}=\CC e_{\alpha}$}\\ \hline
$\g_{\alpha}$, $\alpha\in\Delta^{im}$&$\varnothing$&$\sup_{\alpha\in\Delta^{im}}\dim\g_{\alpha}<\infty$&$\sup_{\alpha\in\Delta^{im}}\dim\g_{\alpha}=\infty$\\ \hline
\end{tabu}}
\end{center}

\begin{example}\label{example:A1tilde1}
In the notations of Example~\ref{example:gA}(2): $\Delta^{im}=\ZZ_{\neq 0}\delta$ with $\delta=\alpha_0+\alpha_1$, and $\g_{n\delta}=\CC\begin{psmallmatrix}t^n&0\\0&-t^n\end{psmallmatrix}$. Similarly, $\Delta^{re}=\{n\delta\pm\alpha_1 \ | \ n\in\ZZ\}$ with $e_{n\delta+\alpha_1}=\begin{psmallmatrix}0&t^n\\0&0\end{psmallmatrix}$ and $e_{n\delta-\alpha_1}=\begin{psmallmatrix}0&0\\-t^n&0\end{psmallmatrix}$.
\end{example}

\begin{remark}\label{remark:locnilp}
Here is another key difference between real and imaginary roots: if $x\in\g_{\alpha}$ is nonzero, then the endomorphism $\ad x\in\End(\g_A)$ is \emph{locally nilpotent} (i.e. for all $y\in\g_A$ there exists $n=n(y)\in\NN$ such that $(\ad x)^ny=0$) if and only if $\alpha\in\Delta^{re}$.
\end{remark}

\section{Kac--Moody groups}\label{section:KMG}

There are many objects deserving the name of Kac--Moody group; the most flexible ``definition'' of a Kac--Moody group is then as follows.
\begin{definition}\label{definition:KMG}
A {\bf Kac--Moody group} is a group attached to a Kac--Moody algebra. More precisely, in order for a group $G$ to deserve the name of Kac--Moody group of type $A$, it should  have an adjoint action $\Ad\co G\to\Aut(\g_A)$ on $\g_A$ (or some variation of $\g_A$, such as a completion) with small, central kernel.
\end{definition}

The most obvious way to construct a Kac--Moody group would then be to exponentiate the adjoint representation $\ad\co\g_A\to\End(\g_A)$. On the other hand, if $\alpha\in\Delta$ and $x\in\g_{\alpha}$ is nonzero, then the map
$$\g_A\to\g_A:y\mapsto (\exp\ad x)y:=\sum_{n\geq 0}\tfrac{1}{n!}(\ad x)^ny$$
is a well-defined automorphism of $\g_A$ if and only if the above sum is always finite, which happens precisely when $\alpha\in\Delta^{re}$ by Remark~\ref{remark:locnilp}. This leads us to define the group
$$G_A^{\ad}(\CC):=\langle \exp\ad x \ | \ x\in\g_{\alpha}, \ \alpha\in\Delta^{re}\rangle=\langle x_{\alpha}(r)=\exp\ad re_{\alpha} \ | \ r\in\CC, \ \alpha\in\Delta^{re}\rangle\leq\Aut(\g_A),$$
which one can call a {\bf minimal\footnote{here, ``minimal'' refers to the fact that we only exponentiate the real root spaces and not the imaginary root spaces.} (adjoint) Kac--Moody group over $\CC$}.

\begin{example}
In the notations of Example~\ref{example:A1tilde1}, we have $G_A^{\ad}(\CC)=\mathrm{PSL}_2(\CC[t,t\inv])$ and $x_{\alpha}(r)=I+re_{\alpha}\in G_A^{\ad}(\CC)$ for all $\alpha\in\Delta^{re}$.
\end{example}

\begin{remark}
If $h\in\h'$ and $y\in\g_{\beta}$, then $(\exp \ad h)y=e^{\beta(h)}y$ also makes sense, and we can thus also exponentiate the adjoint action of $\h'$, to get a {\bf torus}
$$T:=\langle e^{\ad h}=\exp\ad h \ | \ h\in\h'\rangle\leq\Aut(\g_A).$$
But as we will see, $T$ is in fact already contained in $G_A^{\ad}(\CC)$.
\end{remark}

\begin{remark}
Let $\lambda\in\h^*$ be dominant integral, i.e. such that $\lambda(\alpha_i^{\vee})\in\NN$ for all $i\in I$. Then the irreducible highest-weight $\g_A$-module $L(\lambda)$ with highest weight $\lambda$ is integrable: the representation $\pi_{\lambda}\co\g_A\to\End(L(\lambda))$ is such that each $\pi_{\lambda}(e_{\alpha})$ ($\alpha\in\Delta^{re}$) is locally nilpotent. One can thus define in a same the minimal Kac--Moody group
$$G_A^{\pi_{\lambda}}(\CC):=\langle \exp\pi_{\lambda}(x) \ | \ x\in\g_{\alpha}, \ \alpha\in\Delta^{re}\rangle\leq\Aut(L(\lambda)).$$
\end{remark}

Since finite-dimensional simple Lie algebras yield groups that can be defined over any field or even (commutative, associative, unital) ring $k$, such as $\SL_n(k)$, we would now like to define Kac--Moody groups over arbitrary fields or even rings. Moreover, we would like to have an intrinsic definition of a Kac--Moody group that does not depend on an ambiant space such as $\Aut(\g_A)$ or $\Aut(L(\lambda))$. In particular, we would like to understand how the various groups $G_A^{\ad}(\CC)$ and $G_A^{\pi_{\lambda}}(\CC)$ constructed above compare to each other.

Note that $G_A^{\ad}(\CC)$ (and $G_A^{\pi_{\lambda}}(\CC)$) is generated by a torus $T$ and by copies $U_{\alpha}:=x_{\alpha}(\CC)\cong (\CC,+)$ of the additive group of $\CC$ for each $\alpha\in\Delta^{re}$. One could then define a Kac--Moody group over $k$ as a free product of groups $\U_{\alpha}(k)\cong (k,+)$ for $\alpha\in\Delta^{re}$ (we then denote by $$x_{\alpha}\co k\to \U_{\alpha}(k):r\mapsto x_{\alpha}(r)$$ the corresponding isomorphism) and of a torus $T_k\cong (k^{\times})^{|I|}$ (we then write $$T_k=\langle r^{\alpha_i^{\vee}} \ | \ i\in I\rangle\cong (k^{\times})^{|I|},$$ with an injective group morphism $k^{\times}\to T_k:r\mapsto r^{\alpha_i^{\vee}}$ for each $i\in I$), which we quotient out by all the relations we observe between these generators inside $\Aut(\g_A)$ and $\Aut(L(\lambda))$ (at least over $\CC$), with the hope to find sufficiently many such relations to get back the groups $G_A^{\ad}(\CC)$ and $G_A^{\pi_{\lambda}}(\CC)$ for $k=\CC$.

\begin{definition}\label{definition:constructiveTF}
Let $k$ be a ring. Define the group $\G_A(k)=T_k*(\ast_{\alpha\in\Delta^{re}}\U_{\alpha}(k))/\textrm{(R0)--(R4)}$ where the relations (R0)--(R4) are as follows. For each $i\in I$ and $r\in k^{\times}$, set $\widetilde{s}_i(r):=x_{\alpha_i}(r)x_{-\alpha_i}(r\inv)x_{\alpha_i}(r)$ and $\widetilde{s}_i:=\widetilde{s}_i(1)$. For all $\alpha,\beta\in\Delta^{re}$, fix a total order on $]\alpha,\beta[_{\NN}:=(\NN_{>0}\alpha+\NN_{>0}\beta)\cap\Delta$.
\begin{enumerate}
\item[(R0)] $[x_{\alpha}(t),x_{\beta}(u)]=\prod_{\gamma=i\alpha+j\beta\in ]\alpha,\beta[_{\NN}}x_{\gamma}(C^{\alpha\beta}_{ij}t^iu^j)$ for all prenilpotent pairs $\{\alpha,\beta\}\subseteq\Delta^{re}$, where the $C^{\alpha\beta}_{ij}$ are given integers; 
\item[(R1)]
$r^{\alpha_i^{\vee}}x_{\alpha_j}(t)r^{-\alpha_i^{\vee}}=x_{\alpha_j}(r^{a_{ij}}t)$ for all $r\in k^{\times}$, $i,j\in I$ and $t\in k$;
\item[(R2)]
$\widetilde{s}_ir^{\alpha_j^{\vee}}\widetilde{s}_i\inv=r^{s_i\alpha_j^{\vee}}$ for all $r\in k^{\times}$ and $i,j\in I$;
\item[(R3)]
$r^{\alpha_i^{\vee}}=\widetilde{s}_i\inv\widetilde{s}_i(r\inv)$ for all $r\in k^{\times}$  and $i\in I$;
\item[(R4)]
$\widetilde{s}_ix_{\gamma}(t)\widetilde{s_i}\inv=x_{s_i\gamma}(t)$ for all $i\in I$, $t\in k$ and $\gamma\in\Delta^{re}$.
\end{enumerate}
The group functor $\G_A\co\Zalg\to\Gr$ is called the {\bf constructive Tits functor} of type $A$. For $\KK$ a field, the group $\G_A(\KK)$ is called the {\bf minimal Kac--Moody group} of (simply connected) type $A$ over $\KK$. 
\end{definition}

The relation (R1) says that the torus $T_k$ normalises each {\bf real root group} $\U_{\alpha}(k)$. The relation (R2) says that the elements $\widetilde{s}_i$ normalise $T_k$ (the notation $s_i\alpha_j^{\vee}$ refers to the dual action of $W\leq\GL(\h^*)$ on $\h$). The relation (R3) implies that $T_k$ is already contained in the subgroup generated by the real root groups. The relation (R4) says that the elements $\widetilde{s}_i$ permute (by conjugation) the real root groups according to the corresponding action of the Weyl group $W$ on $\Delta^{re}$. 

We now explain (R0). Note that the product $\prod_{\gamma=i\alpha+j\beta\in ]\alpha,\beta[_{\NN}}x_{\gamma}(C^{\alpha\beta}_{ij}t^iu^j)$ only makes sense if the interval $]\alpha,\beta[_{\NN}$ is finite and contained in $\Delta^{re}$. A pair of roots $\{\alpha,\beta\}\in\Delta^{re}$ is {\bf prenilpotent} if there exist $v,w\in W$ such that $\{v\alpha,v\beta\}\subseteq\Delta^{re+}$ and $\{w\alpha,w\beta\}\subseteq\Delta^{re-}$. If $\beta\neq\pm\alpha$, this turns out to be equivalent to requiring that $]\alpha,\beta[_{\NN}$ be finite and contained in $\Delta^{re}$.

Finally, we explain where the constants $C^{\alpha\beta}_{ij}\in\ZZ$ come from. Using (R4) and the definition of a prenilpotent pair, we may assume that $\alpha,\beta\in\Delta^{re+}$. We would like to be able to write down the exponential $\exp(re_{\alpha})=\sum_{n\geq 0}r^ne_{\alpha}^n/n!$ for $r\in k$ ($k$ a ring) in a suitable space. A natural candidate would be the enveloping algebra of $\n^+$, except we need to define a $k$-form of this algebra in which the fractions $1/n!$ make sense, and which also allows for infinite sums.

\begin{definition}\label{definition:Zform}
Let $\UU_{\CC}(\g_A)$ be the universal enveloping algebra of $\g_A$, and consider its $\ZZ$-subalgebra $\UU^+_{\ZZ}$ generated by the elements $e_i^n/n!$ for $i\in I$ and $n\in\NN$. Then $\UU^+_{\ZZ}$ is a $\ZZ$-form of $\UU_{\CC}(\n^+)$, that is, the canonical map $\UU^+_{\ZZ}\otimes_{\ZZ}\CC\to\UU_{\CC}(\n^+)$ is an isomorphism. For a ring $k$, one also defines the $k$-form $\UU^+_k:=\UU^+_{\ZZ}\otimes_{\ZZ}k$. Let $\UU^+_{k}=\bigoplus_{\alpha\in Q_+}\UU^+_{\alpha k}$ be the $Q_+$-gradation of $\UU^+_k$ inherited from the $Q_+$-gradation of $\n^+$, where $Q_+:=\bigoplus_{i\in I}\NN\alpha_i$ (note that the $\UU^+_{\alpha k}$ are finite-dimensional $k$-modules). Finally, set $\widehat{\UU}^+_k:=\prod_{\alpha\in Q_+}\UU^+_{\alpha k}$.

For each $\alpha\in\Delta^{re+}$ we have $e_{\alpha}^n/n!\in\UU^+_{n\alpha\ZZ}$, so that for any $r\in k$ the exponential
$$\exp(re_{\alpha})=\sum_{n\geq 0}r^ne_{\alpha}^n/n!\in (\widehat{\UU}^+_k)^{\times}$$
belongs to the group of invertible elements of $\widehat{\UU}^+_k$.
\end{definition}

The integers $C^{\alpha\beta}_{ij}$ can then be computed from the group commutator $[g,h]:=g\inv h\inv gh$ of the exponentials $\exp(te_{\alpha})$ and $\exp(ue_{\beta})$ inside $\widehat{\UU}^+_k$.

\begin{example}
Suppose that the simple roots $\{\alpha_i,\alpha_j\}$ form a subsystem of type $A_2$, that is, the corresponding sub-matrix of $A$ is $\begin{psmallmatrix}2&-1\\ -1&2\end{psmallmatrix}$. We then have $]\alpha_i,\alpha_j[_{\NN}=\{\alpha_i+\alpha_j\}$ and $e_{\alpha_i}=e_i$, $e_{\alpha_j}=e_j$ and $e_{\alpha_i+\alpha_j}=[e_i,e_j]$. Using the Serre relations $(\ad e_i)^2e_j=0=(\ad e_j)^2e_i$, one can compute in $\widehat{\UU}^+_k$ that
$$[\exp(te_{\alpha_i}),\exp(ue_{\alpha_j})]=[\exp(te_i),\exp(ue_j)]=\exp(tu[e_i,e_j])=\exp(tue_{\alpha_i+\alpha_j}),$$
so that $C^{\alpha_i\alpha_j}_{11}=1$.
\end{example}

Now that we have explained Definition~\ref{definition:constructiveTF}, we make two remarks, respectively justifying why $\G_A(\KK)$ is the right object for $A$ of finite type, and why it is not too small (the presentation does not collapse) over any ring $k$. 

\begin{remark}
If $A$ is a Cartan matrix and $\KK$ a field, then $\G_A(\KK)$ is the Chevalley group of (simply connected) type $A$.
\end{remark}

\begin{remark}\label{remark:GAadk}
One can extend $\UU^+_{\ZZ}$ (see Definition~\ref{definition:Zform}) to a $\ZZ$-form $\UU_{\ZZ}$ of $\UU_{\CC}(\g_A)$, and hence for any ring $k$ define a $k$-form $\g_{Ak}:=(\g_A\cap\UU_{\ZZ})\otimes_{\ZZ}k$ of $\g_A$. As before, we can then define the group
$$G_A^{ad}(k):=\langle \exp\ad re_{\alpha} \ | \ r\in k, \ \alpha\in\Delta^{re}\rangle\leq\Aut(\g_{Ak}).$$
One easily checks that the relations (R0)--(R4) are satisfied in $G_A^{ad}(k)$ (this is how the relations were found!), and so we have an adjoint action map
$$\Ad_k\co\G_A(k)\twoheadrightarrow G_A^{ad}(k)\leq\Aut(\g_{Ak}).$$
In particular, $\G_A(k)$ is ``big enough'' as it admits $G_A^{ad}(k)$ as a quotient (and the same can be done with the highest weight representations $\pi_{\lambda}$ instead of $\ad$).
\end{remark}

To justify why $\G_A(\KK)$ is not too big either over fields $\KK$ (i.e. we found sufficiently many relations so that $\Ad_{\KK}$ has small, central kernel), we will need the theory of buildings (see the ``Prerequisites'' notes).

\section{Buildings}

If $\KK$ is a field, it readily follows from the relations (R0)--(R4) that the real root groups $\U_{\alpha}(\KK)$ satisfy the axioms of an RGD system (which is no surprise as RGD systems were defined to fit this picture).

\begin{lemma}
Let $\KK$ be a field. Then $(\G_A(\KK),(\U_{\alpha}(\KK))_{\alpha\in\Delta^{re}},T_{\KK})$ is an RGD system of type $(W,S)$.
\end{lemma}

The deeper result of this section, which is purely building-theoretic, is that RGD systems yield (twin) BN-pairs, and hence strongly transitive actions on buildings.

\begin{corollary}
Let $\KK$ be a field. Set $$\U^{\pm}_A(\KK)=\langle \U_{\alpha}(\KK) \ | \ \alpha\in\Delta^{re\pm}\rangle, \quad B^{\pm}=T_{\KK}\ltimes U^{\pm}_A(\KK)\quad\textrm{and}\quad N:=\langle\widetilde{s}_i, \ T_{\KK} \ | \ i\in I\rangle.$$
Then $(B^+,N)$ and $(B^-,N)$ are (twinned, saturated) BN-pairs. In particular, $\G_A(\KK)$ acts strongly transitively on the associated buildings $X_{\pm}=X(\G_A(\KK),B^{\pm})$, and we have Bruhat decompositions $\G_A(\KK)=\coprod_{w\in W}B^{\pm}wB^{\pm}$.
\end{corollary}

\section{Axiomatic}

At this point, we could already justify why $\G_A(\KK)$ is not too big, but before we do so, let us also ask ourselves the question whether this group $\G_A(\KK)$ is essentially unique, or in other words whether any group deserving the name of ``(minimal) Kac--Moody'' in the sense of Definition~\ref{definition:KMG} is already isomorphic to $\G_A(\KK)$.

Such deserving candidates should be group functors $\G\co\Zalg\to\Gr$ such that $\G(\CC)$ has an adjoint action on $\g_A$ with small, central kernel. In other words, they should come equipped with group functor morphisms $\varphi_i\co\SL_2\to\G$ ($i\in I$) and $\eta\co T\to \G$ (respectively exponentiating the fundamental copies $\CC e_i+\CC \alpha_i^{\vee}+\CC f_i$ of $\sll_2(\CC)$ and the Cartan subalgebra $\h'$), such that
\begin{enumerate}
\item[(KMG5)] there is an adjoint action $\Ad_{\CC}\co \G(\CC)\to \Aut(\g_A)$, with central kernel contained in $T_{\CC}$, such that 
$$\Ad_{\CC}\varphi_i\begin{psmallmatrix}1 &r\\ 0&1\end{psmallmatrix}=\exp\ad re_i, \quad \Ad_{\CC}\varphi_i\begin{psmallmatrix}1 &0\\ -r&1\end{psmallmatrix}=\exp\ad rf_i\quad\textrm{and}\quad \Ad_{\CC}(\eta(e^{r\alpha_i^{\vee}}))=\exp\ad r\alpha_i^{\vee}$$ for all $r\in\CC$ and $i\in I$.
\end{enumerate} 
Since we want our group to be ``minimal'', it should also be generated by the fundamental copies of $\SL_2$ and the torus (at least over fields):
\begin{enumerate}
\item[(KMG1)] If $\KK$ is a field, $\G(\KK)$ is generated by the $\varphi_i(\SL_2(\KK))$ ($i\in I$) and $\eta(T(\KK))$. 
\end{enumerate} 
Note that these two axioms are not sufficient: over $\CC$, the group $\G(\CC)$ is small enough thanks to (KMG5), but for $k\neq\CC$, we could take for $\G(k)$ the free product of the $\varphi_i(\SL_2(k))$ ($i\in I$) and $\eta(T(k))$ without violating (KMG1) or (KMG5). To ensure that $\G(k)$ is also small enough for $k\neq\CC$ (at least over fields), we then impose one last axiom:
\begin{enumerate}
\item[(KMG4)] If $k\to\CC$ is an injective morphism from a ring $k$ to $\CC$, then the corresponding group morphism $\G(k)\to\G(\CC)$ is also injective. 
\end{enumerate} 

\begin{definition}
A triple $(\G,(\varphi_i)_{i\in I},\eta)$ as above satisfying the axioms (KMG1), (KMG4), (KMG5) is called a {\bf Tits functor} of type $A$.
\end{definition}

Here is now the desired uniqueness statement. Let $\Ga\co\Zalg\to\Gr:k\mapsto (k,+)$ be the additive group functor, and let $x_{\pm}\co\Ga\to\SL_2$ be the functors defined by
$$x_+(r)=\begin{psmallmatrix}1 &r\\ 0&1\end{psmallmatrix}\quad\textrm{and}\quad x_-(r)=\begin{psmallmatrix}1 &0\\ -r&1\end{psmallmatrix}.$$

\begin{theorem}\label{thm:uniquenessTF}
Let $(\G,(\varphi_i)_{i\in I},\eta)$ be a Tits functor of type $A$. Then there is a unique morphism of group functors $\pi\co\G_A\to\G$ such that the diagrams
\begin{equation*}
\xymatrix{
T \ar[r] \ar[rd]_{\eta} &
\G_{A} \ar[d]^{\pi}\\
&\G}\qquad\quad\textrm{and}\qquad\quad
\xymatrix{
\Ga \ar[r]^{x_{\pm \alpha_i}} \ar[rd]_{x_{\pm}} &
\U_{\pm\alpha_i} \ar[r] &
\G_{A} \ar[d]^{\pi}\\
&\SL_2\ar[r]_{\varphi_i} & \G}
\end{equation*}
are commutative. Moreover, $\pi_{\KK}\co\G_A(\KK)\to\G(\KK)$ is an isomorphism for any field $\KK$ (up to kernel contained in $T_{\KK}$), provided $\varphi_i(\SL_2(\KK))\not\subseteq\pi_{\KK}(\U^+_A(\KK))$ for all $i\in I$.\footnote{This last condition is a ``non-degeneracy condition'' that prevents degenerate examples.}
\end{theorem}
\begin{proof}
Here is a sketch proof of this theorem.
\begin{enumerate}
\item
By (KMG5), we have the morphism $\pi_{\CC}\co\G_A(\CC)\to\G(\CC)$ (as (KMG5) says that essentially, $\G(\CC)$ is just $G^{ad}_A(\CC)$).
\item
By (KMG4), we then get the morphism $\pi_k\co\G_A(k)\to\G(k)$ for any subring $k$ of $\CC$. 
\item
Using the functoriality of $\G_A$ and $\G$, this then yields the desired morphism $\pi_k$ for any ring $k$.
\item
Assume now that $k=\KK$ is a field. Then $\pi_{\KK}$ is surjective by (KMG1). Moreover, $\pi_{\KK}$ maps the RGD system of $\G_A(\KK)$ to an RGD system of $\G(\KK)$ of the same type\footnote{This is where the non-degeneracy condition is used.}. In particular, the kernel of
$$\pi_{\KK}\co \G_A(\KK)=\coprod_{w\in W}B^+wB^+\to\G(\KK)=\coprod_{w\in W}\pi_{\KK}(B^+)w\pi_{\KK}(B^+)$$
lies in $B^+$. Hence $\mathrm{Ker}(\pi_{\KK})\subseteq\bigcap_{g\in G}gB^+g\inv=T_{\KK}$, as desired (where the last equality follows from the fact that the BN-pair $(B^+,N)$ is saturated). \qedhere
\end{enumerate}
\end{proof}

\begin{remark}\label{remark:CTFTF}
Note that, as shown by the above proof, we need for the uniqueness statement over fields that the Tits functors be defined over rings and not just fields: otherwise, there would for instance be no way to go from $\CC$ to a finite field when trying to construct the morphism $\pi$.

We also note that the terminology is a bit confusing, in that the constructive Tits functor $\G_A$ is (probably) not a Tits functor, as it (probably) does not satisfy (KMG4) in full generality (see also Section~\ref{subsection:Injectivity}).
\end{remark}

As the group functor $\G=G^{ad}_A$ from Remark~\ref{remark:GAadk} trivially satisfies the assumptions of Theorem~\ref{thm:uniquenessTF}\footnote{Technically speaking, one has to replace the group functors $\SL_2$ by the elementary subgroup functors $\SSL_2$.}, we can now justify why $\G_A(\KK)$ is not too big for $\KK$ a field.
\begin{corollary}
Let $\KK$ be a field. Then the adjoint action map $\Ad_{\KK}\co\G_A(\KK)\twoheadrightarrow G^{ad}_A(\KK)\leq\Aut(\g_{A\KK})$ has kernel contained in $T_{\KK}$.
\end{corollary}

Since Kac--Moody algebras of affine type are realised as (twisted) loop algebras over a finite-dimensional simple Lie algebra, Theorem~\ref{thm:uniquenessTF} also allows to identify the corresponding minimal Kac--Moody groups over fields.

\begin{corollary}
Let $A=\begin{psmallmatrix}2&-2\\ -2&2\end{psmallmatrix}$. Then $\G_A(\KK)\cong\SL_2(\KK[t,t\inv])\rtimes\KK^{\times}$ for $\KK$ a field.
\end{corollary}

\section{Maximal Kac--Moody groups}

Let $A=(a_{ij})_{i,j\in I}$ be a GCM and $\KK$ be a field. So far, we have constructed a ``minimal'' Kac--Moody group $\G_A(\KK)$ by exponentiating the real root spaces of the Kac--Moody algebra $\g_A$, and have seen that $\G_A(\KK)$ admits the following action maps with kernel contained in $T_{\KK}$:
\begin{enumerate}
\item
$\G_A(\KK)\to\Aut(\g_{A\KK})$;
\item
$\G_A(\KK)\to\Aut(L_{\KK}(\lambda))$ for $\lambda\in\h^*$ dominant integral (where $L_{\KK}(\lambda)$ is a $\KK$-form of $L(\lambda)$ defined using $\UU_{\ZZ}$);
\item
$\G_A(\KK)\to\Aut(X_{\pm})$ where $(X_+,X_-)$ is the twin building of $\G_A(\KK)$.
\end{enumerate}

Going back to the beginning of Section~\ref{section:KMG}, we could try a bit harder to define an ``adjoint'' Kac--Moody group over $\CC$ by also exponentiating the adjoint action of imaginary root spaces on $\g_A$. The problem was that if $\alpha\in\Delta^{im}$ and $x\in\g_{\alpha}$ is nonzero, then $(\exp\ad x)y=\sum_{n\geq 0}\tfrac{1}{n!}(\ad x)^ny$ is in general an infinite sum for $y\in\g_A$. To remedy, this, we could try to complete $\g_A$ to allow such infinite sums: recall that 
$$\g_A=\bigoplus_{\alpha\in\Delta\cup\{0\}}\g_{\alpha}=\n^-\oplus\h'\oplus\n^+.$$
If we allow for infinite sums of homogeneous elements, say $x=\sum_{\alpha}x_{\alpha}$ and $y=\sum_{\beta}y_{\beta}$ with $x_{\beta},y_{\beta}\in\g_{\beta}$, then we need to be able to define the Lie bracket $[x,y]$, which for each $\gamma\in\Delta$ has homogeneous component of degree $\gamma$ given by
$$[x,y]_{\gamma}=[\sum_{\alpha}x_{\alpha},\sum_{\beta}y_{\beta}]_{\gamma}=\sum_{\alpha+\beta=\gamma}[x_{\alpha},x_{\beta}].$$
But this last sum only makes sense if it is finite, and this prevents us from simultaneously allowing infinitely many nonzero homogeneous components of positive and negative degrees. In other words, we have to choose a direction in which to complete, positive or negative: we set
$$\widehat{\g}_A:=\n^-\oplus\h'\oplus\widehat{\n}^+\quad\textrm{where}\quad \widehat{\n}^+:=\prod_{\alpha\in\Delta^+}\g_{\alpha}.$$
We could then define\footnote{The proper way to do it is in fact to take a closure of that group, see Definition~\ref{definition:completions}.}
$$G_{A,ad}^{\max}(\CC):=\langle \exp\ad x \ | \ x\in\g_{\alpha}, \ \alpha\in\Delta^{re}\cup\Delta^{im+}\rangle\leq\Aut(\widehat{\g}_A).$$
More conceptually, for any field $\KK$, we could define a completion of $\G_A(\KK)$ inside each of the spaces $\Aut(\g_{A\KK})$, $\Aut(L_{\KK}(\lambda))$ and $\Aut(X_+)$.

\begin{definition}\label{definition:completions}
We define the completions of (the image of) $\G_A(\KK)$ inside each the following spaces, with respect to the topology of uniform convergence on bounded sets\footnote{The correct definition of these three completions of $\G_A(\KK)$ is actually a slight modification of the completions defined here, in order to ensure that the torus $T_{\KK}$ of $\G_A(\KK)$ injects in each of them.}:
\begin{enumerate}
\item inside $\Aut(\widehat{\g}_{A\KK})$: this is the {\bf algebraic completion} of $\G_A(\KK)$, denoted $\G_A^{\mathrm{alg}}(\KK)$;
\item inside $\Aut(L_{\KK}(\lambda))$: this is the {\bf representation-theoretic completion} of $\G_A(\KK)$, denoted $\G_A^{\mathrm{rt}\lambda}(\KK)$;
\item
inside $\Aut(X_+)$: this is the {\bf geometric completion} of $\G_A(\KK)$, denoted $\G_A^{\mathrm{geo}}(\KK)$.
\end{enumerate}
The metric on each of the spaces $\widehat{\g}_{A\KK}$, $L_{\KK}(\lambda)$ and $X_+$ that give a sense to ``bounded'' sets (and hence to the topology on their automorphism group) is as follows. For $X_+$, the metric is the chamber distance. Since $\widehat{\g}_{A\KK}$ and $L_{\KK}(\lambda)$ are $\ZZ$-graded vector spaces (for the gradation induced by root height), they are also equiped with a natural metric, in which two vectors are close if their difference is a sum of homogeneous elements of high degree.
\end{definition}

\begin{example}
If $A=\begin{psmallmatrix}2&-2\\ -2&2\end{psmallmatrix}$, then one can check\footnote{up to a difference at the level of the torus.} that the completions of $\G_A(\KK)\cong\SL_2(\KK[t,t\inv])\rtimes\KK^{\times}$ coincide with $\SL_2(\KK(\!(t)\!))\rtimes\KK^{\times}$.
\end{example}

\begin{example}
If $\KK=\FF_q$ is a finite field, then these completions are totally disconnected locally compact (tdlc) groups, which are locally pro-$p$ where $p=\charact(\KK)$.
\end{example}

As for minimal Kac--Moody groups, we would also like to define a more intrinsic completion of $\G_A(\KK)$, i.e. one that does not depend on an ambiant space. With the experience of Section~\ref{section:KMG}, it seems to be a good idea to look at exponentials $\exp(x)$ of homogeneous elements $x\in\g_{\alpha k}$ with $\alpha\in\Delta^{+}$ inside $\widehat{\UU}^+_{k}$, for each ring $k$. 

\begin{prop}
Let $k$ be a ring. For each $\alpha\in\Delta^+$, let $\BBB_{\alpha}$ be a $\ZZ$-basis of $\g_{\alpha\ZZ}:=\g_{\alpha}\cap \UU^+_{\ZZ}$, and fix a total order on $\BBB=\bigcup_{\alpha\in\Delta^+}\BBB_{\alpha}$. Set\footnote{The elements $[\exp](\lambda x)$ are called \emph{twisted exponentials}: as we have seen, when $x\in\BBB_{\alpha}$ with $\alpha\in\Delta^{re+}$, the divided powers $x^n/n!$ belong to $\UU^+_{\ZZ}$ and hence the exponentials $\exp(\lambda x)=\sum_{n\geq 0}\lambda^nx^n/n!$ make sense in $\widehat{\UU}^+_k$. However, this is no longer true if $\alpha\in\Delta^{im+}$: in that case, if $k$ is not of characteristic zero, we have to replace exponentials by these twisted exponentials, which are in some sense the ``best possible approximations'' of exponentials living inside $\widehat{\UU}^+_k$.}
$$\U^{ma+}_A(k):=\Big\{\prod_{x\in\BBB}[\exp](\lambda_x x)\in \widehat{\UU}^+_k \ | \ \lambda_x\in k\Big\}\subseteq \widehat{\UU}^+_k.$$
Then the following assertions hold:
\begin{enumerate}
\item $\U^{ma+}_A(k)$ is a subgroup of $(\widehat{\UU}^+_k)^{\times}$. In fact, the group functor $\U^{ma+}_A\co\Zalg\to\Gr$ is even an affine group scheme.
\item Each element $g\in\U^{ma+}_A(k)$ has a unique expression $g=\prod_{x\in\BBB}[\exp](\lambda_x x)$ with $\lambda_x\in k$.
\item
We have a group morphism $\U^+_A(k)\to\U^{ma+}_A(k):x_{\alpha}(\lambda)\mapsto \exp(\lambda e_{\alpha})$, which is injective if $k$ is a field.
\item
The sets $$\U^{ma}_n(k):=\Big\{\prod_{x\in\BBB, \ \height(\deg(x))\geq n}[\exp](\lambda_x x)\in \widehat{\UU}^+_k \ | \ \lambda_x\in k\Big\}$$ for $n\in\NN$ are normal subgroups of $\U^{ma+}_A(k)$ and form a basis of identity neighbourhoods for a complete Hausdorff group topology on $\U^{ma+}_A(k)$.
\item
If $\charact \KK=0$ or $\charact\KK>M_A:=\max_{i\neq j}|a_{ij}|$, then $\U^+_A(\KK)$ is dense in $\U^{ma+}_A(\KK)$.
\end{enumerate}
\end{prop}

We then define an intrinsic completion of $\G_A(k)$ in the spirit of Definition~\ref{definition:constructiveTF}.

\begin{definition}
For $k$ a ring, we define the {\bf scheme-theoretic completion} $\G_A^{\mathrm{sch}}(k)$ of $\G_A(k)$ via a presentation
$$\G_A^{\mathrm{sch}}(k)=\G_A(k)*\U^{ma+}_A(k)/\textrm{(obvious relations)}.$$
We then have a natural inclusion $\G_A(\KK)\hookrightarrow \G_A^{\mathrm{sch}}(\KK)$ and $\overline{\G_A(\KK)}=\G^{\mathrm{sch}}_A(\KK)$ as soon as $\charact \KK=0$ or $\charact\KK>M_A$.
\end{definition}

An analogue of the uniqueness statement for maximal Kac--Moody groups is still out of reach in general, but we nevertheless have the following.

\begin{theorem}\label{thm:GKsimplicity}
\begin{enumerate}
\item The indentity map on $\G_A(\KK)$ induces continuous group morphisms\footnote{There is a subtlety for the second arrow when $0<\charact\KK<M_A$.}
$$\overline{\G_A(\KK)}\to \G_A^{\mathrm{alg}}(\KK)\to\G_A^{\mathrm{rt}\lambda}(\KK)\to \G_A^{\mathrm{geo}}(\KK).$$
\item
If $\charact\KK=0$ and $A$ is symmetrisable, these are isomorphisms of topological groups.
\item
If $\KK=\FF_q$ is a finite field, these are surjective, and are isomorphisms if and only if\footnote{or rather, if and only if $\overline{\G_A(\KK)}$ is GK-simple.} $\G_A^{\mathrm{sch}}(\KK)$ is {\bf GK-simple}, in the sense that the kernel
$$Z'(\G_A^{\mathrm{sch}}(\KK)):=\bigcap_{g\in \G_A^{\mathrm{sch}}(\KK)}g(T_{\KK}\U^{ma+}_A(\KK))g\inv$$
of the $\G_A^{\mathrm{sch}}(\KK)$-action on the building $X_+$ is contained in $T_{\KK}$.
\end{enumerate}
\end{theorem}

\begin{remark}\label{remark:indgroupscheme}
One can also construct a {\bf (positive) maximal Kac--Moody ind-group scheme\footnote{that is, $\G_A^{\mathrm{pma}}$ is a group functor which is an inductive limit of affine schemes.}} $\G_A^{\mathrm{pma}}\co\Zalg\to\Gr$ such that
\begin{enumerate}
\item
there is a natural morphism $\G_A^{\mathrm{sch}}\to\G_A^{\mathrm{pma}}$ of group functors which is an isomorphism of topological groups over fields.
\item
if $A$ is a Cartan matrix, then $\G_A^{\mathrm{pma}}$ is the Chevalley--Demazure affine group scheme of type $A$.
\end{enumerate}
\end{remark}

\section{Open questions}\label{section:openpb}
We collect here a few open questions pertaining to the foundations of the theory. More details can be found in Chapter~8 and 9 of the book.

\subsection{Injectivity of the Tits functor}\label{subsection:Injectivity}

First, as noted in Remark~\ref{remark:CTFTF}, it is unclear how far the constructive Tits functor is from being a Tits functor because of the axiom (KMG4). This leads to the following question.

\begin{question}
Given a domain $k$ with field of fractions $\KK$, when is the natural map $\G_A(k)\to\G_A(\KK)$ injective? 
\end{question}

This question can also be stated for arbitrary rings as follows:

\begin{question}
Given a ring $k$, when is the natural map $\G_A(k)\to\G^{\mathrm{pma}}_A(k)$ injective? 
\end{question}

For $A$ of finite type, a lot of work has been done in that direction (the keywords being ``K2-theory for Chevalley groups'').

\subsection{GK-simplicity}
As we saw, having a uniqueness statement for maximal Kac--Moody groups over fields $\KK$ (at least when $\charact \KK=0$ or $\charact\KK>M_A$) essentially amounts to establishing that $\G_A^{\mathrm{sch}}(\KK)$ is GK-simple.

\begin{question}
When is $\G_A^{\mathrm{sch}}(\KK)$ GK-simple?
\end{question}

Note that when $0<\charact\KK<M_A$, this is in general false, but the hope is that $\G_A^{\mathrm{sch}}(\KK)$ would be GK-simple whenever $\charact\KK>M_A$.

When $\charact\KK=0$, the group $\G_A^{\mathrm{sch}}(\KK)$ is GK-simple if $A$ is symmetrisable by Theorem~\ref{thm:GKsimplicity}(2). For $A$ non-necessarily symmetrisable, this is equivalent to proving the Gabber--Kac simplicity Theorem~\ref{thm:GKac} for $A$ (whence the terminology ``Gabber--Kac simple'' or ``GK-simple'').

When $\KK=\FF_q$ is a finite field, GK-simplicity amounts to all the completions of $\G_A(\KK)$ being isomorphic (see Theorem~\ref{thm:GKsimplicity}(3)).

Beyond aiming for a uniqueness statement for maximal Kac--Moody groups over fields, there is another motivation to study the GK-simplicity of $\G_A^{\mathrm{sch}}(\KK)$: it amounts to understand the kernel of the map
$$\G_A^{\mathrm{sch}}(\KK)\to\G^{\mathrm{geo}}_A(\KK)\leq\Aut(X_+),$$
and hence the difference between, on the one hand, the group $\G^{\mathrm{geo}}_A(\KK)$ which has the nice property of being simple (see \S\ref{subsection:simplicity}), and on the other hand the group $\G_A^{\mathrm{sch}}(\KK)$ whose fine algebraic structure is much more easily accessed thanks to the connection between $\U^{ma+}_A(\KK)$ and $\g_{A\KK}$. In practise, this means that the GK-simplicity question is often in the way when trying to prove fundamental properties of maximal Kac--Moody groups (see also \S\ref{subsection:linearity} and \ref{subsection:isompb} for illustrations of this).

On the other hand, when $0<\charact\KK<M_A$, all hell seems to break loose (non-GK simplicity, non-density of $\G_A(\KK)$ in $\G_A^{\mathrm{sch}}(\KK)$, exceptional isomorphisms, etc), and it would be nice to understand what exactly is going on.

\begin{question}
What is happening in small characteristic???
\end{question}

\subsection{Linearity}\label{subsection:linearity}
A group $G$ is \emph{linear} if there exists a group morphism $\varphi\co G\to\GL_n(F)$ for some field $F$, with central kernel. If $A$ is of finite or affine type, then the groups $\G_A(\KK)$ and $\G_A^{\mathrm{sch}}(\KK)$ are linear. On the other hand, as soon as $A$ is of indefinite type, the group $\G_A(\KK)$ is not linear (with some exceptions for $\KK=\FF_q$ and $|I|=2$). The following question, on the other hand, is still open:

\begin{question}
Suppose $A$ is of indefinite type. When are the groups $\U^+_A(\KK)$ and $\U^{ma+}_A(\KK)$ (non-)linear?
\end{question}

For instance, it is a long-standing open question whether $\U^{ma+}_A(\FF_q)$ can be linear over a local field $F$. Note that if $\G_A^{\mathrm{sch}}(\FF_q)$ were known to be GK-simple, then this question can be answered (by the negative, for $A$ of indefinite type).

\subsection{Simplicity}\label{subsection:simplicity}
Suppose that $A$ is indecomposable.
It is known that $\G_A^{\mathrm{sch}}(\KK)$ is (abstractly) simple modulo its Gabber--Kac kernel for most fields $\KK$ (including fields of characteristic zero and finite fields). For the minimal Kac--Moody group $\G_A(\KK)$, things are more complicated: it is known that $\G_A(\FF_q)$ is simple modulo center provided $A$ is not of affine type (and with additional mild assumptions on $|I|$ and $q$), but the following question remains open:

\begin{question}
Suppose $A$ is of indefinite indecomposable type. Is $\G_A(\KK)$ simple modulo centre when $\KK$ is a field of characteristic zero?
\end{question}

\subsection{Isomorphism problem}\label{subsection:isompb}
Another natural question is the isomorphism problem for Kac--Moody groups. The minimal Kac--Moody group $\G_A(\KK)$ determines $A$ when $\KK$ has characteristic zero, but when $\KK=\FF_q$ is a finite field, there are exceptional isomorphisms and $\G_A(\KK)$ turns out to carry very little information about $A$. 

One can ask a similar question for maximal Kac--Moody groups:

\begin{question}
Does the topological group $\G_A^{\mathrm{sch}}(\KK)$ determine $A$?
\end{question}

As the exceptional isomorphisms between minimal Kac--Moody groups over $\FF_q$ extend to topological isomorphisms of the corresponding geometric completions, it is important to consider the scheme-theoretic completion $\G_A^{\mathrm{sch}}(\KK)$ rather than another completion of $\G_A(\KK)$, as it is the completion that seems to carry the most information about $A$, even over finite fields.

Again, partial results to answer the above question can be obtained if we knew the groups $\G_A^{\mathrm{sch}}(\KK)$ were GK-simple.

\subsection{Characterisations of Kac--Moody groups}
Finally, as mentioned in Section~\ref{section:KMA}, there are no known ``realisations'' of Kac--Moody algebras (and hence groups) when $A$ is of indefinite type. It would be nice to be able to give examples of Kac--Moody groups of indefinite type, in the same way one can give $\SL_n(\KK[t,t\inv])$ and $\SL_n(\KK(\!(t)\!))$ as examples of Kac--Moody groups of affine type. More importantly, we would like to be able to characterise Kac--Moody groups as ``building blocks'' of various theories, thereby underlining their fundamental nature. This raises the following question, which we then decline in three sub-questions respectively pertaining to the theories of schemes, locally compact groups, and buildings.

\begin{question}
Can we characterise (a class of) Kac--Moody groups without reference to Kac--Moody algebras?
\end{question}

It is known that the Chevalley--Demazure affine group schemes (which maximal Kac--Moody ind-group schemes generalise by Remark~\ref{remark:indgroupscheme}(2)) are precisely the (split, semisimple) affine group schemes over $\ZZ$. 

\begin{question}
Are maximal Kac--Moody ind-group schemes the (\emph{suitable adjectives}) ind-group schemes over $\ZZ$?
\end{question}

The quotient of $\G_A^{\mathrm{sch}}(\FF_q)$ by its Gabber--Kac kernel belongs to the class $\mathscr{S}$ of compactly generated, (topologically) simple, non-discrete tdlc groups. This class $\mathscr{S}$ plays a fundamental role in the general structure theory of locally compact groups (with suitable non-discreteness assumptions). Here are two facts about $\mathscr{S}$\footnote{For more details, we refer to the survey paper ``\href{http://dx.doi.org/10.4171/176-1/15}{Non-discrete simple locally compact groups}'' by Pierre-Emmanuel Caprace.}:
\begin{enumerate}
\item
To each group $G$ in $\mathscr{S}$ one can attach in a canonical way a compact space $\Omega_G$ on which $G$ acts nicely.
\item
The only known examples of groups $G$ in $\mathscr{S}$ such that $\Omega_G$ is trivial (i.e. reduced to a point) are simple algebraic groups over local fields and (conjecturally) locally compact Kac--Moody groups (namely, maximal Kac--Moody groups over finite fields).
\end{enumerate}
In other words, locally compact Kac--Moody groups are the only known examples (at least conjecturally) of groups $G$ in $\mathscr{S}$ that are non-linear and with trivial $\Omega_G$.

\begin{question}
Can we characterise locally compact Kac--Moody groups within the class $\mathscr{S}$?
\end{question}

Finally, recall that any minimal Kac--Moody group $\G_A(\KK)$ acts on a twin building $(X_+,X_-)$. When $A$ is \emph{$2$-spherical} (i.e. $a_{ij}a_{ji}\leq 3$ for all $i\neq j$), the automorphism group $\Aut(X_+,X_-)$ (that is, the group of automorphisms of $X_+\times X_-$ preserving the twinning) is not ``too big'' compared to $\G_A(\KK)$ (or rather, it is too big if $A$ is not $2$-spherical).

\begin{question}
Assume that $A$ is $2$-spherical, with corresponding Weyl group $(W,S)$. Can we characterise minimal Kac--Moody groups $\G_A(\KK)$ as automorphism groups of twin buildings of type $(W,S)$ (with suitable extra assumptions) ?
\end{question}
 
The answer to this last question is probably already known.\footnote{See the paper ``\href{http://dx.doi.org/10.1017/CBO9780511629259.023}{Twin buildings and groups of Kac--Moody type}'' by Jacques Tits, as well as subsequent work of Bernhard Mühlherr.}

\bibliographystyle{amsalpha} 
\bibliography{overviewKM}

\end{document}